\documentclass{article}

\usepackage[square,numbers]{natbib}
\usepackage[utf8]{inputenc}
\usepackage{amssymb}
\usepackage{amsthm}
\usepackage{amsmath}
\usepackage{amsfonts}
\usepackage{mathtools}

\usepackage{microtype}
\usepackage{float}
\usepackage{hyperref}
\hypersetup
{
	colorlinks,
	linkcolor = blue,
	citecolor = blue,
	urlcolor = blue,
	filecolor = blue,
	breaklinks,
}

\usepackage{color}
\usepackage{graphicx}
\graphicspath{ {images/} }

\usepackage{algpseudocode}
\usepackage{algorithm}
\MakeRobust{\Call}

\usepackage[accepted]{icml2018}

\newtheorem{lemma}{Lemma}[section]
\newtheorem*{example}{Example}
\newtheorem{thm}{Theorem}[section]

\theoremstyle{definition}
\newtheorem{definition}{Definition}[section]
\theoremstyle{remark}
\newtheorem*{remark}{Remark}

\DeclareMathOperator{\im}{im}
\DeclareMathOperator{\st}{st}
\DeclareMathOperator{\lk}{lk}
\DeclareMathOperator{\degr}{deg}

\newcommand{\reducedhom}{\tilde{H}}
\newcommand*\cl[1]{\overline{#1}}

\newcommand{\vr}[2]{{\text{VR}_{#1}(#2)}}

\newcommand{\transpose}[1]{#1^\top}
\newcommand{\norm}[1]{\left\lVert#1\right\rVert}
\newcommand{\wordv}[1]{v_\text{#1}}
\DeclarePairedDelimiterX\setc[2]{\{}{\}}{\,#1 \;\delimsize\vert\; #2\,}
\newcommand\restr[2]{{
		\left.\kern-\nulldelimiterspace 
		#1 
		\vphantom{\big|} 
		\right|_{#2} 
}}
\newcommand{\zmod}[1]{\mathbb{Z}/#1\mathbb{Z}}

\icmltitlerunning{Local Homology of Word Embeddings}

\begin{document}

\twocolumn[
\icmltitle{Local Homology of Word Embeddings}

\icmlsetsymbol{equal}{*}

\begin{icmlauthorlist}
\icmlauthor{Tadas Tem\v{c}inas}{oxf}
\end{icmlauthorlist}

\icmlaffiliation{oxf}{This work was done as part of a master thesis at the University of Oxford}
\icmlcorrespondingauthor{Tadas Tem\v{c}inas}{temcinas@gmail.com}
\icmlkeywords{topological data analysis, natural language processing, local homology}
\vskip 0.3in
]

\printAffiliationsAndNotice{\icmlEqualContribution}


\begin{abstract}
Topological data analysis (TDA) has been widely used to make progress on a number of problems. However, it seems that TDA application in natural language processing (NLP) is at its infancy. In this paper we try to bridge the gap by arguing why TDA tools are a natural choice when it comes to analysing word embedding data. We describe a parallelisable unsupervised learning algorithm based on local homology of datapoints and show some experimental results on word embedding data. We see that local homology of datapoints in word embedding data contains some information that can potentially be used to solve the word sense disambiguation problem.
\end{abstract}


\section{Introduction}

Topological data analysis (TDA) has become a very active and broad area of research, that uses tools from topology to analyse data. In particular, persistent homology seems to be the TDA tool of choice for many applications \citep{roadmap_persistence, oudot}. Tools that will be discussed in this work are those leveraging local topological structure, more concretely -- local homology. On the practical side, implementations of such tools very often can be easily parallelisable.

Word embedding is a collective term for ways to represent words of a natural language as vectors in a high-dimensional real vector space. We will see in the following chapters that datasets coming from word embeddings have interesting topology. Hence natural language processing (NLP) seems to be a natural domain of TDA applications. Despite this fact, to the best of our knowledge, only a few attempts at using TDA techniques to analyse language data have been published \citep{topology_NLP1, topology_NLP2, topology_NLP3, topology_NLP4}. One of the aims of this work is to contribute to closing the gap between TDA and NLP by applying TDA techniques to NLP data.

Intuitively, stratification is a decomposition of a topological space into manifold-like pieces. When thinking about stratification learning and word embeddings, it seems intuitive that vectors of words corresponding to the same broad topic would constitute a structure, which we might hope to be a manifold. Hence, for example, by looking at the intersections between those manifolds or singularities on the manifolds (both of which can be recovered using local homology-based algorithms \citep{nanda_strat}) one might hope to find vectors of homonyms like `bank' (which can mean either a river bank, or a financial institution). This, in turn, has potential to help solve the word sense disambiguation (WSD) problem in NLP, which is pinning down a particular meaning of a word used in a sentence when the word has multiple meanings.

In this work we present a clustering algorithm based on local homology, which is more relaxed\footnote{meaning that on the same data the produced clusters is the result on merging the clusters produced by the other algorithms} that the stratification-based clustering discussed in \citep{nanda_strat} and the sheaf-theoretic clustering from \citep{brown_sheaf_decomposition}, where the sheaf is taken to be the local homology sheaf. We implement the algorithm and make the code available \citep{code}, and finally we discuss the results on two datasets of word vectors that come from pre-trained word embeddings. We find that the first local homology groups of some word vectors can reflect the fact that the same token has different meanings (e.g. `bank'). The second local homology groups in this context can be interpreted as counting the number of different `similarity loops' in the neighbourhood of a word vector.

\section{Local homology}

Let us remind ourselves that given a CW complex $X$ we can define a partial order on its cells. By definition of a CW complex we know that if $\sigma, \, \tau$ are cells such that $\sigma \cap \cl{\tau} \neq \emptyset$, then we have $\sigma \subseteq \cl{\tau}$ and we say that $\sigma$ is a face of $\tau$, and $\tau$ is a co-face of $\sigma$. Write $\sigma \leq \tau$ iff $\sigma$ is a face of $\tau$. This relation defines a partial order on the set of cells of $X$. Now we can define the star and link of a cell by $\st(\sigma) = \setc*{\tau \in X}{\sigma \leq \tau}$ and $\lk(\sigma) = \setc*{\tau \in X}{\exists \rho \in X. \: \tau, \, \sigma \leq \rho \wedge \nexists \rho' \in X. \: \rho' \leq \tau, \, \sigma}$. 

As usual, we write $H_\bullet(X; G)$ for the homology of $X$ with coefficients in an abelian group $G$ and $H_\bullet(X, A; G)$ for the relative homology with coefficients in $G$. If the group $G$ is not specified, it means that $G = \mathbb{Z}$.



\begin{definition}\label{def_localhom}
	Let $X$ be a CW complex and let $\sigma$ be a cell in $X$. The local homology of $\sigma$ in $X$ is $H_\bullet^\sigma = H_\bullet(\cl{\st(\sigma)}, \partial \st(\sigma))$.
	Rank of the free part of $H_n^\sigma$ is called the $n$-th local betti number of $\sigma$ in $X$.
\end{definition}

\begin{definition}
	A CW complex $X$ is called regular iff the attaching map for each cell is a homeomorphism.
\end{definition}

The lemmas below are stated in the setting of finite simplicial complexes since in the end when we compute homology for the applications, we use them. However, the analogous lemmas hold in a more general setting of finite regular CW complexes.

We remind ourselves that given a finite simplicial complex $X$ we can orient it by defining a total order on the vertices $\{v_0, v_1, \ldots, v_N\}$.\footnote{We do this when calculating simplicial homology.} We can specify an $n$-simplex $\sigma$ by an $(n+1)$-tuple of vertices $(v_{\sigma_0}, v_{\sigma_1}, \ldots, v_{\sigma_n})$ where $\sigma_i < \sigma_{i+1}$. Also, let $(v_{\sigma_0}, v_{\sigma_1}, \ldots,  \hat{v_{\sigma_i}}, \ldots, v_{\sigma_n})$ be the face of $\sigma$ specified by the vertices $\sigma \setminus \{v_{\sigma_i}\}$. Now to calculate local homology of a simplex, we can take star of the simplex, pretend that it is a simplicial complex and then calculate its simplicial homology. This is formulated more precisely in the following lemma:

\begin{lemma} \label{combinatorial_localhom}
	Take a simplex $\sigma$ in a finite simplicial complex $X$. Let $S_n(\sigma) = \setc*{\tau \subseteq \sigma}{\left\vert{\tau}\right\vert = n+1}$. Let $C_n^\sigma = \mathbb{Z}[S_n]$ and define $\partial_n: C_n^\sigma \to C_{n-1}^\sigma$ on the basis elements: $\partial_n(\tau) = \sum_{i=0}^{n}(-1)^i(v_{\tau_0}, v_{\tau_1}, \ldots,  \hat{v_{\tau_i}}, \ldots, v_{\tau_n})$. Set $(v_{\tau_0}, v_{\tau_1}, \ldots,  \hat{v_{\tau_i}}, \ldots, v_{\tau_n}) = 0$ iff the simplex specified by $(v_{\tau_0}, v_{\tau_1}, \ldots,  \hat{v_{\tau_i}}, \ldots, v_{\tau_n})$ is not in $S_{n-1}$. Then we have that $(C_\bullet^\sigma, \partial)$ is a chain complex and $H_\bullet(C_\bullet^\sigma) = H_\bullet^\sigma$.
\end{lemma}

	

\begin{remark}
	Instead of taking the integral homology in the Definition \ref{def_localhom} we could take homology with coefficients in an abelian group $G$ and then we would get a definition of local homology with coefficients in $G$. In this setting the previous lemma holds just as well if we take $C_n^\sigma = \mathbb{Z}[S_n] \otimes G$.
\end{remark}

The following lemma shows that if we think about a simplicial complex as a generalisation of a graph (on the basis that 1-skeleton of a simplicial complex is an undirected graph without self-loops or multiple edges between the same vertices), then we can think about local homology as a generalisation of a degree of a vertex. Hopefully this will provide some intuition of how local homology works.

\begin{lemma}\label{lochom_as_deg}
	Let $X$ be a 1-dimensional finite simplicial complex. For any 0-simplex $v \in X$ let $\degr(v) = \left\vert{\setc*{e \in X}{v \in e}}\right\vert$. Then for any 0-simplex $v$: $$H_n^v = \left\{
	\begin{array}{ll}
	\mathbb{Z}^{\degr(v) - 1} & \quad n=1\\
	0 & \quad \text{otherwise}
	\end{array}
	\right.$$
\end{lemma}

	
	

The following lemma characterises the local homology in terms of the link of a simplex instead of its star. We will use this viewpoint when analysing the computational results on data.

\begin{lemma}\label{lem:local_hom_link}
	If $X$ is a finite simplicial complex and $\sigma$ is a $k$-simplex, then for $n \geq k + 1$ we have $H_n^\sigma = \reducedhom_{n-k-1}(\lk(\sigma))$.
\end{lemma}

\begin{remark}
	The above lemma also holds if we calculate local homology with coefficients in an abelian group $G$. See apendix for details.
\end{remark}

\begin{example}
	Let $X_k$ be a simplicial complex that is $k$ triangles glued on a common edge. $X_k$ is the simplicial closure of $\setc*{\{0,1,i+2\}}{i \in \{0, 1, \ldots, k-1\}}$. In Figure \ref{fig:x3} we can see $X_3$. Let $\sigma := \{0,1\}$ be the common edge. Then $\lk(\sigma) = \{\{2\}, \{3\}, \ldots, \{k+1\}\}$ consists of $k$ 0-simplices. In this case it is easy to calculate the homology of the link - it has $k$ connected components hence 0th homology will be $\mathbb{Z}^k$. Since there are no other simplices but the 0-dimensional ones, all the higher homology groups are trivial. Therefore by Lemma \ref{lem:local_hom_link} we have $H_2^\sigma = \mathbb{Z}^{k-1}$ and all the other local homology groups being trivial.
\end{example}

\begin{figure}[H]
	\caption{$X_3$}
	\centering
	\includegraphics[width=2cm, height=3cm]{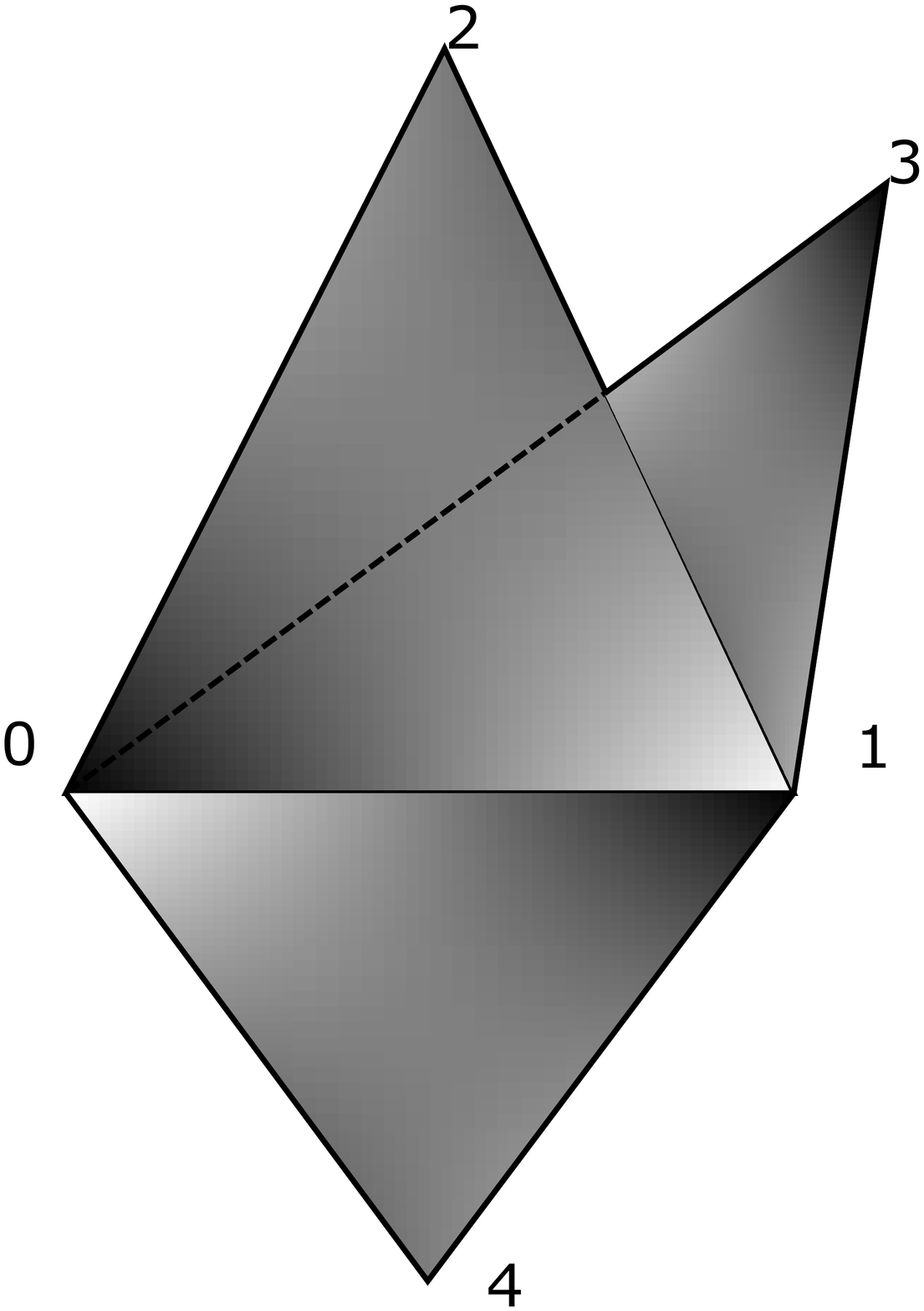}
	\label{fig:x3}
\end{figure}

\section{Word embeddings: an overview}

Here we quickly recap word embeddings by concentrating on two models: Word2Vec (skip-gram version) and GloVe. For an intuitive discussion of word embeddings we refer to \citep{colah}. For an academic introduction into word embeddings (and much more), we refer to \citep{nlp_intro}.


Word embeddings seems to be an umbrella term, which refers to various models that take a corpus of text data as an input, and output a function from the set of words in the corpus to $\mathbb{R}^d$ for some $d \in \mathbb{N}$.\footnote{For the pre-trained models that we will be using, $d$ is set to $300$, which is quite typical.} For the purpose of this work, a more formal framework for the study of word embeddings can be given by the following definitions:

\begin{definition}
	Given a finite alphabet $\Sigma$, the set of all possible words in $\Sigma$ is $\Sigma^* = \setc*{x_1 x_2 \ldots x_n}{x_i \in \Sigma}$.
\end{definition}

For the purpose of our discussion we can assume that $\Sigma$ contains all lowercase and uppercase letters from the English alphabet.

\begin{definition}
	A corpus $C$ is a finite sequence of words $\{w_1, w_2, \ldots, w_T\}$
\end{definition}

\begin{definition}
	Vocabulary $V$ of a corpus $C$ is the set of all words that appear in $C$.
\end{definition}

\begin{definition}
	A word embedding with respect to vocabulary $V$ is a function $E: V \to \mathbb{R}^d$.
\end{definition}

\begin{remark}
	According to our definition, distributional semantics models (DSM) are also word embeddings. In the context of this work, since we treat those models as a black box, there is no difference between DSMs and the more conventional notion of word embeddings. In fact, even if we do not treat them as black boxes, some of the word embedding models and DSMs are mathematically equivalent as shown in \citep{we_and_dsm}, questioning this distinction.
\end{remark}

\subsection{Word2Vec model}

Word2Vec is a collection of similar models that given a corpus produce word embeddings. The models were developed in \citep{word2vec, wor2vec_implementation}. Here we describe a particular variation from \citep{wor2vec_implementation}, which we later use in our experiments.

The idea behind this variation (called the skip-gram model) of Word2Vec is to come up with a word embedding that given a word $w_i$ in a corpus $C$ would predict well for some $n \in \mathbb{N}$ the $n$ words before $w_i$ and the $n$ words after $w_i$ (called context of the word $w_i$). More formally, we can think of this using the following definitions.

Assume we have a corpus $C = \{w_1, w_2, \ldots, w_T\}$, its vocabulary $V$ and a desired dimension of a word embedding we are trying to construct $d$.

\begin{definition}
	Let functions $i: V \to \mathbb{R}^d$ and $o: V \to \mathbb{R}^d$ be called input and output functions respectively.
\end{definition}

\begin{definition}
	Given input and output functions $i$ and $o$ define the following quantity for any $v, w \in V$: $$p(v \vert w) = \frac{\exp(\transpose{o(v)} i(w))}{\sum \limits_{w' \in V}\exp(\transpose{o(w')} i(w))}$$ which will later be interpreted as the probability of observing a word $v$ given that one has observed a word $w$.
\end{definition}

\begin{definition}\label{word2vec_cost}
	Given input and output functions $i$ and $o$ as well as $n \in \mathbb{N}$ define the cost quantity: $$ \text{cost}_n = -\frac{1}{T} \sum \limits_{t=1}^{T} \sum \limits_{-n \leq j \leq n, j \neq 0} \log(p(w_{t+j} \vert w_t))$$
\end{definition}

\begin{definition}
	A skip-gram model, given an input function $i: V \to \mathbb{R}^d$ (usually initialised to some random function) and some $n \in \mathbb{N}$, constructs an output function $o: V \to \mathbb{R}^d$ so that the cost quantity $\text{cost}_n$ is minimised. This output function $o$ is the function that one then uses as a word embedding.
\end{definition}

Here we do not discuss the implementation details of this model but refer anyone interested to \citep{wor2vec_implementation}.



\subsection{GloVe model}

GloVe is a model developed in \citep{glove}. It is known to be a model which is derived (almost) uniquely from the desired properties of a word embedding. We refer to \citep{glove} for full discussion discussion.

To define the model, assume we have a corpus $C = \{w_1, w_2, \ldots, w_T\}$, its vocabulary $V$ and a desired dimension of a word embedding that we are trying to construct $d$.

Like in the skip-gram model case, we will have the parameter $n \in \mathbb{N}$ which will determine how big the context of a word is. In \citep{glove}, when implementing the model $n = 10$ is chosen.

\begin{definition}
	For any $v, w \in V$ let $c(v,w)$ be the number of times the word $w$ appears in the context of $v$ in the corpus. Encoding each word by a natural number, we can put the quantities $c(v,w)$ into a matrix, which is called the word co-occurrence matrix.
\end{definition}

\begin{remark}
	It is possible to introduce some weights to the context quantity $c(-,-)$ described above to reflect the intuition that words in the context of $w$ that are further away from $w$ are less significant.
\end{remark}

\begin{definition}
	Let $v_\text{context}, \, v_\text{word}: V \to \mathbb{R}^d$ be the context-vector and word-vector functions respectively.
\end{definition}

\begin{definition}
	Let $b_\text{context}, \, b_\text{word}: V \to \mathbb{R}$ be the context-bias and word-bias functions respectively.
\end{definition}

\begin{definition}[\citep{glove}]
	A function $f: \mathbb{R} \to \mathbb{R}$ is called a GloVe weighting function iff it satisfies the following properties:
	\begin{itemize}
		\item $f(0) = 0$ and $f$, viewed as a continuous function, satisfies $\lim_{x \to 0}f(x)\log^2(x) = 0$.
		\item $f$ is non-decreasing.
		\item $f(x)$ is ``relatively small for large values of $x$''.\footnote{This property is obviously not mathematically well-defined but it is an intuitive requirement proposed in \citep{glove}}
	\end{itemize}
\end{definition}

\begin{remark}
	In \citep{glove} the following function is used as a GloVe weighting function:
	$$f(x) = \left\{
	\begin{array}{ll}
	(\frac{x}{100})^\frac{3}{4} & \quad x < 100\\
	1 & \quad \text{otherwise}
	\end{array}
	\right.$$
\end{remark}

\begin{definition}
	Given $n \in \mathbb{N}$ and hence the notion of the context, a GloVe weighting function $f$ and functions $v_\text{context}, \, v_\text{word}: V \to \mathbb{R}^d$, and $b_\text{context}, \, b_\text{word}: V \to \mathbb{R}$ the cost quantity of the GloVe model is defined to be: 
	\begin{equation}
	\begin{split}
	\text{cost}_\text{GloVe} = \sum_{w,w'\in V}f(c(w,w'))[\transpose{v_\text{word}(w)} v_\text{context}(w') \\ 
	+ b_\text{word}(w) + b_\text{context}(w') - \log(c(w,w'))]^2
	\end{split}
	\end{equation}
\end{definition}

\begin{definition}
	For fixed $n, \, d \in \mathbb{N}$ and a GloVe weighting function $f$, the GloVe model constructs $v_\text{context}, \, v_\text{word}: V \to \mathbb{R}^d$, and $b_\text{context}, \, b_\text{word}: V \to \mathbb{R}$ such that $\text{cost}_\text{GloVe}$ is minimised. One then uses $v_\text{context} + v_\text{word}$ as a word embedding.
\end{definition}

\begin{remark}
	In fact one could use $v_\text{context}$ or $v_\text{word}$ alone but it was noticed empirically in \citep{glove} that they do not differ much and combining them additively yields slightly better results probably because the model is less prone to overfitting.
\end{remark}

\section{Topology and word embeddings}

Here we discuss the interesting properties that the word embeddings produced by GloVe and skip-gram models have and why topology might be a good framework for studying such embeddings. We will see that word embeddings, even though trained on lexical data with no additional information of semantics and syntax, can learn semantical and syntactical information quite well.

\begin{definition}\label{def:geodesic_distance}
	Let $S^{d-1} := \setc*{x \in \mathbb{R}^d}{\norm{x} = 1}$ be the $(d-1)$-dimensional unit sphere embedded into the Euclidean space $\mathbb{R}^d$, which is equipped with the usual dot product. Define $d_\text{geo}: S^{d-1} \times S^{d-1} \to \mathbb{R}_{\geq 0}$ by setting for all $x, y \in S^{d-1} \; d_\text{geo}(x,y) = \cos^{-1}(x \cdot y)$, where $x \cdot y$ is the usual dot product in $\mathbb{R}^d$. From basic topology we know that $(S^{d-1}, d_\text{geo})$ defines a metric space and we call it the unit sphere with geodesic distance.
\end{definition}

Firstly, let us set the scene for thinking about the word embeddings in the language of topology. Let $f: V \to \mathbb{R}^d$ be a word embedding of a vocabulary $V$.\footnote{The reader can assume $f$ is produced by one of the models we have discussed. However, the following discussion holds for a larger class of models.} Naturally, we want to look at $f(V)$, which is a point cloud in $\mathbb{R}^d$. Empirically we know that various semantical relationships are best captured by looking at the angle between different (linear combinations of) elements of $f(V)$ \citep{mikolov_regularities, levy_regularities, glove}. 

Motivated by this, we will map $f(V)$ into the $(d-1)$-dimensional unit sphere with the geodesic distance by taking each vector in $f(V)$ and dividing it by its norm. We will obtain a new set: $$W = \setc*{\frac{v}{\norm{v}} \in S^{d-1}}{v \neq 0, \, v \in f(V)}$$ Note that it is possible to have two vectors $v,w \in f(V)$ such that $\frac{v}{\norm{v}} = \frac{w}{\norm{w}}$ but we do not observe such situations in practice. We will write $\wordv{word}$ for the vector in $W$ corresponding to the word `word'. 

The interesting properties of both skip-gram and GloVe models were observed on many occasions and a short summary of them can be found in \citep[Subsections 4.1, 4.3]{glove}. The most relevant properties for our work is good performance on word analogy and word similarity tasks. A word analogy task is of the form ``a is to b like c is to what?''. A typical example is ``queen is to king like man is to what?'' and the correct answer is `woman'. A task ``a is to b like c is to what?'' can be answered by looking at the vector $\wordv{a} - \wordv{b} + \wordv{c}$ and then finding a vector $\wordv{d} \in W$ which is the closest to $\wordv{a} - \wordv{b} + \wordv{c}$ with respect to the geodesic distance. The word `d' would be returned as the answer to the word analogy task. Depending on corpora that models are trained on and the set of analogies that is being used for testing, word embeddings can achieve accuracy of up to $81.9 \%$ on semantic and $69.3 \%$ on syntactic analogies \citep[Table 2]{glove}.

\begin{figure}
	\caption{Taken from \citep{mikolov_regularities}.}
	\centering
	\includegraphics[width=0.25\textwidth]{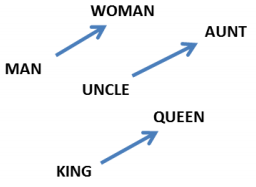}
	\label{fig:gender_vecs}
\end{figure}

Perhaps a more important property of word embeddings from a topological point of view is that they map similar words to vectors that are close with respect to the geodesic distance. This can be tested by taking pairs of words and assigning them with scores, given by a human, on how similar the words in a single pair are, and then seeing how the scores predicted by the model (i.e. the (cosine of) geodesic distance of the pair) differ from the human-produced ones. Again, a high accuracy up to $83.6 \%$ is observed \citep[Table 3]{glove}.

On the basis of the performance on word analogy and similarity tasks, we see that the datasets coming from word embeddings have a potential of exhibiting interesting topology.

\section{Local homology clustering}

Here we present a clustering algorithm based on local homology and review the results on two datasets that we generate from embeddings of a certain set of words with respect to two different embeddings.

In the previous section we have seen that word embedding vectors of similar words end up close to each other when they are mapped into a sphere equipped with the geodesic distance. Hence, we hope that similar words comprise some sort of structure. Continuing this line of thought, we hope that by looking at a local neighbourhood of some word vectors we can distinguish words with vastly different meanings since they would lie at the intersection of two (or more) structures of word vectors of similar words. For example, we might hope that the vector of `bank' would be in the structure of vectors of words related to finance and also in the structure of vectors of words related to rivers since `bank' can mean either a certain financial institution or a river bank. Detecting such words would help to solve the word sense disambiguation problem in NLP. Local homology seems a natural tool and hence it is explored here.

\subsection{Algorithm and pipeline}

In the light of this discussion it seems that recovering canonical stratification of a simplicial complex coming from a dataset is exactly what we are looking for. However, we found that on our data full stratification learning does not seem to provide us with interesting information because the clusters become very discrete (i.e. we observe very few words clustered together). Hence we consider a weaker clustering algorithm. We require a local homology preserving path between vertices of a simplicial complex to cluster them, which is a necessary condition of them belonging to the same canonical stratum of the simplicial complex, as proved in \citep{nanda_strat}. Even with this relaxed condition, we will see that the clusters themselves carry little information. However, the local homology itself seems to be very informative.

Our pipeline is as follows:
\begin{enumerate}
	\item Start with a point cloud $S$ in a metric space $(M,d)$.
	\item Pick $\epsilon$ and build the Vietoris-Rips (VR) complex associated to $S$ and $\epsilon$ -- $\vr{\epsilon}{S}$.
	\item Since each $s \in S$ is a 0-simplex in $\vr{\epsilon}{S}$, the notion of the local homology of $s$ is well-defined, so we compute $H_\bullet^s$ for each $s \in S$.
	\item For each edge (i.e. 1-simplex) $e$ in $\vr{\epsilon}{S}$, we also compute $H_\bullet^e$.
	\item Define a relation on $S$: $s \sim p$ iff $H_n^s \cong H_n^p$ for all $n \in \mathbb{N}$ and there is a path of edges $\{e_1, e_2, \ldots, e_k\}$ between $s$ and $p$ such that for any $n \in \mathbb{N}$ and for any $i,j \in \{1, 2, \ldots, k\}$ we have $H_n^{e_i} \cong H_n^{e_j} \cong H_n^s \cong H_n^p$. Note that $\sim$ defines an equivalence relation.
	\item We output equivalence classes in $S$ with respect to $\sim$ (together with their local homology groups) and call them clusters.
\end{enumerate}

We provide pseudocode for the less trivial parts of the pipeline. From now on assume that $S$ is a point cloud in a metric space $(M, d)$ and we have picked $\epsilon \in \mathbb{R}_{\geq 0}$ and $d$, which will serve as a cut-off point for the VR-complex.

For simplicity, we present the algorithms that calculate local homology over $\zmod{2}$. However, $\zmod{2}$ can be replaced by any other field and the algorithms would still work. In fact, the algorithms can be easily extended to calculate local homology over the ring $\mathbb{Z}$ but we would need to keep track of the torsion and not just the betti numbers as we do now.

The first step of the pipeline is to build the VR-complex. The VR-complex can have higher dimensionality than the ambient space that the data is coming from, so it is usual to restrict the complex and look at its $d$-skeleton. For computation of the VR-complex we use the incremental algorithm as described in \citep[Section 4.2]{fast_VR}.

Let us discuss an algorithm that computes $k$-th local homology with coefficients in $\zmod{2}$. We assume that a simplicial complex that we are dealing with is totally ordered as a set -- arbitrary order is sufficient. Also, in the following algorithm we use a notion of empty matrix. It is a degenerate case of a matrix -- it has 0 rows and 0 columns as well as 0 entries. We also assume to have the following subroutines:
\begin{itemize}
	\item $\Call{SmithNF}$ -- given a matrix $M$ with entries in $\zmod{2}$ it returns the Smith Normal Form of $M$.
	\item $\Call{ZeroCol}$ and $\Call{NZeroCol}$ -- given a matrix $M$, return the number of zero columns in $M$ and the number of non-zero columns in $M$ respectively.
\end{itemize}

\begin{algorithm*}[h]
	\caption{Local Homology}
	\begin{algorithmic}[1]
		\Procedure {Star}{$K$, $\tau$} 
		\Comment{$K$ is a finite simplicial complex, $\tau \in K$}
		\State \Return $\setc*{\sigma \in K}{\tau \subseteq \sigma}$
		\Comment{Inherits ordering from $K$}
		\EndProcedure
		\Statex

		\Procedure {GetBoundaryOperator}{$K$, $k$}
		\State $B_\text{domain} := \setc*{\sigma \in K}{\left\vert{\sigma}\right\vert = k+1} = \{d_1, \ldots, d_n\}$
		\Comment{Ordering inherited from $K$}
		\State $B_\text{codomain} := \setc*{\sigma \in K}{\left\vert{\sigma}\right\vert = k} = \{c_1, \ldots, c_m\}$
		\Comment{Ordering inherited from $K$}
		\If {$B_\text{domain} = \emptyset$}
		\State \Return empty matrix
		\ElsIf {$k=0$}
		\State \Return $1 \times \left\vert{B_\text{domain}}\right\vert$ matrix with all zero entries
		\Else
		\State $M_{i,j} :=  \left\{
		\begin{array}{ll}
		1 & \quad d_j \subseteq c_i\\
		0 & \quad \text{otherwise}
		\end{array}
		\right.$
		\State \Return $[M_{i,j}]_{m \times n}$
		\EndIf
		\EndProcedure
		\Statex

		
		\Procedure {LocalHomology}{$K$, $\tau$, $k$}
		\State $s_\tau := \Call{Star}{K, \tau}$
		\State $\partial_k := \Call{GetBoundaryOperator}{s_\tau, k}$
		\State $\partial_{k+1} := \Call{GetBoundaryOperator}{s_\tau, k+1}$
		\State $\text{dim}(\ker \partial_k) := \Call{ZeroCol}{\Call{SmithNF}{\partial_k}}$
		\State $\text{dim}(\im \partial_{k+1}) := \Call{NZeroCol}{\Call{SmithNF}{\partial_{k+1}}}$
		\State \Return $\text{dim}(\ker \partial_k) - \text{dim}(\im \partial_{k+1})$
		\EndProcedure
	\end{algorithmic}
\end{algorithm*}

\begin{thm}
	Given a finite simplicial complex $K$ and $\tau \in K$, the procedure $\Call{LocalHomology}{K, \tau, k}$ computes $\dim H_k^\tau$.
\end{thm}

\begin{proof}
	Since we are working over $\zmod{2}$ and with finite simplicial complexes, the local homology groups are just finite-dimensional vector spaces over $\zmod{2}$ and hence are characterised by their dimension.
	
	Here we are just invoking Lemma \ref{combinatorial_localhom}. It is clear that $\partial_\bullet$ in lines 18-19 are the boundary operators from Lemma \ref{combinatorial_localhom}. It is a fact from basic algebra that the number of zero-columns in the Smith Normal Form of a matrix associated with a linear map $f$ is the dimension of $\ker f$ and the number of non-zero columns in the dimension of $\im f$. Hence $\text{dim}(\ker \partial_k)$ and $\text{dim}(\im \partial_{k+1})$ are calculated correctly in lines 20, 21.
	
	By Lemma \ref{combinatorial_localhom}, $H_k^\tau = \ker \partial_k / \im \partial_{k+1}$ and hence we have $\dim H_k^\tau = \text{dim}(\ker \partial_k) - \text{dim}(\im \partial_{k+1})$ as returned by the procedure.
\end{proof}

The final step of the pipeline is to cluster points in $S$ according to the equivalence relation defined at the beginning of this subsection. Assume we have a subroutine $\Call{DFS}{V, E}$ -- given a graph $G = (V, E)$ it returns the connected components of $G$ using the Depth First Search algorithm.

\begin{algorithm}[h]
	\caption{Local Homology Clustering}
	\begin{algorithmic}[1]
		\Procedure {Skeleton}{$K$, $k$}
		\Comment{$K$ - finite simplicial complex, $k \in \mathbb{N}$}
		\State \Return $\setc*{\tau \in K}{\left\vert{\tau}\right\vert \leq k}$
		\EndProcedure
		\Statex
		\Procedure {Cluster}{$S$, $\epsilon$, $d$}
		\State $\mathcal{V} := \Call{IncrementalVR}{S, \epsilon, d}$
		\State $\mathcal{V}_1 := \Call{Skeleton}{\mathcal{V}, 1}$
		\ForAll {$\tau \in \mathcal{V}_1$}
		\ForAll {$0 \leq k \leq d$}
		\State $H_k^\tau = \Call{LocalHomology}{\mathcal{V}, \tau, k}$
		\EndFor
		\EndFor
		\State $V:= \Call{Skeleton}{\mathcal{V}, 0}$
		\State $E:= \setc*{\{v,w\} \in \mathcal{V}_1}{H_\bullet^v \cong H_\bullet^w \cong H_\bullet^{\{v,w\}}}$
		\State \Return $\Call{DFS}{V,E}$
		\EndProcedure
	\end{algorithmic}
\end{algorithm}

\begin{thm}
	$\Call{Cluster}{S, \epsilon, d}$ computes the clustering of $S$ defined at the beginning of the subsection.
\end{thm}

\begin{proof}
	Since in line 14 we are removing all the edges that do not preserve the local homology between their endpoints, what we are left with are the edges that do preserve local homology. Therefore all paths in such a graph will be local homology preserving and so by definition of the equivalence relation, connected components of this graph will be exactly the required equivalence classes.
\end{proof}

\subsection{Implementation}

We implement this algorithm in Python 3.6. Coefficients in $\zmod{2}$ are used to calculate the local homology groups. We implement everything from scratch except the $\Call{SNF}$ subroutine, which we take from \citep{SNF} and adapt to our needs by making the implementation iterative rather than recursive. We use a popular linear algebra library called \textit{numpy} to do the matrix calculations and the built-in Python library \textit{multiprocessing} to implement distribution across multiple processors. The code can be found in a GitHub repository \citep{code}.

\section{Results on data}

We run the local homology clustering algorithm on datasets obtained by taking an embedding of 155 words that are either related to the topic of water and rivers or to the topic of finance. There are a few words that are not related to any of the two topics. The reason for this choice is the hope to see interesting local structure near the word `bank', which is a homonym, related to both rivers and finance.

We use two pre-trained word embeddings:
\begin{enumerate}
	\item Embedding obtained by training the skip-gram model on ``part of Google News dataset (about 100 billion words). The model contains 300-dimensional vectors for 3 million words and phrases'' \citep{pretrained_word2vec}.
	\item Embedding obtained by training the GloVe model on web crawl data obtained from Common Crawl. The result is 300-dimensional word embedding \citep{pretrained_glove} and we use the spaCy library in Python to access it.
\end{enumerate}

By taking the word vectors of the 155 words with respect to the first and the second embeddings, and regarding them as lying on the unit sphere with geodesic distance (Definition \ref{def:geodesic_distance}), we get respective point clouds $D_\text{skip-gram}$ and $D_\text{GloVe}$, which are used as an input to the pipeline.

Also, we note that we have performed calculations without mapping the word vectors to the unit sphere and using the Euclidean distance. However, no interesting results have been found; almost all datapoints end up having all homology groups being 0 or that of a point. As we will see, this is not the case when using the geodesic distance (Definition \ref{def:geodesic_distance}) and so this again suggests that the geodesic distance induces more interesting topological properties than the Euclidean distance.

Even though the two datasets exhibit different local structure and local homologies of vectors corresponding to the same words are very different, we see similar structure of clusters obtained:
\begin{itemize}
	\item For higher values of $\epsilon$ there is a giant cluster with local homology all being zero, which we can think of as the boundary of the dataset. All the other clusters mostly have only one point in them. Very rarely we see 2 or 3 points clustered together.
	\item For lower values of $\epsilon$, 2 or 3 point clusters are even more rare, there is no giant cluster, almost all clusters have only one point in them.
\end{itemize}
The reason for such clustering results is possibly not the similarity between the datasets but the fact that the requirement of local homology groups to be isomorphic is very limiting and in practice we observe that local homology groups are vastly different. Therefore, it is rare to see interesting clusters. This can be changed by relaxing the isomorphism requirement by looking at persistent local homology instead. This would take out the $\epsilon$ parameter but would introduce another parameter $d$ for the distance between two barcodes in the barcode space. Then we could say that persistent local homology groups are ``isomorphic'' (i.e. sufficiently similar to be clustered together)  iff at each homological degree the distance between the barcodes is less than $d$. However, this is left as a part of the future work.

Because the clusters do not have interesting structure, they are not discussed further. Now we will explore the local homology groups of different datapoints.

\subsection{Results on $D_\text{skip-gram}$}

On this dataset we have performed experiments for all $\epsilon \in \{40^\circ, 41^\circ, \ldots, 81^\circ\}$.

Note that since we are calculating homology over $\zmod{2}$, each homology group is completely characterised by its betti number. Hence, for simplicity purposes we denote (local) homology groups as a list of natural numbers. Writing $a_0 \, a_1 \, \ldots \, a_n$ for $a_i \in \mathbb{N}$ means that $H_i = (\zmod{2})^{a_i}$ for $i \in \{0, 1, \ldots, n\}$ and $H_i = 0$ for $i > n$. Table \ref{word2vec_localhom} contains a selection of words, of which their word vectors have local homology that we find interesting. 

\begin{table}[h]
	\centering
	\caption{Words with interesting homology ($D_\text{skip-gram}$)}
	\begin{tabular}{||c c c||} 
		\hline
		Word & Value of $\epsilon$ & Local homology \\ [0.5ex] 
		\hline\hline
		bank &	80 &	0 0 1 5 \\
		bank &	79 &	0 1 1 2 \\
		bank&	78&	0 1 1 \\
		bank&	75&	0 1 3 \\
		bank&	74&	0 1 2 \\
		bank&	66&	0 5 \\
		corporation&	81&	0 0 0 4 4 \\ 
		corporation	&80&	0 0 2 6 \\
		corporation	&79&	0 0 5 4 \\
		corporation	&78&	0 0 5\\
		corporation	&77&	0 0 10\\
		corporation	&76&	0 0 7\\
		corporation	&75&	0 2 1\\
		invest&	81&	0 0 0 7\\
		invest&	80&	0 0 1 5\\
		invest&	79&	0 0 5\\
		invest&	78&	0 0 6\\
		invest&	74&	0 2 1\\
		invest&	73&	0 2 1\\
		manufacturing&	81&	0 0 0 4 1\\
		manufacturing&	80&	0 0 1 3\\
		market&	80&	0 1 0 2\\
		market&	79&	0 1 2 1\\
		market&	77&	0 1 1\\
		market&	76&	0 1 1\\
		savings&	80&	0 0 3 1\\
		savings&	77&	0 1 1\\
		savings&	75&	0 1 1\\
		transaction&	81&	0 2 2 1\\
		transaction&	79&	0 1 4\\
		transaction&	78&	0 1 4\\[1ex] 
		\hline
	\end{tabular}
	\label{word2vec_localhom}
\end{table}

Here we define our notation. We will write $S$ for the point cloud $D_\text{skip-gram}$, $\vr{\epsilon}{S}$ for the VR-complex with parameter $\epsilon$ of the point cloud $S$ and $G_\epsilon(S)$ for the 1-skeleton of $\vr{\epsilon}{S}$. Also, we say that $x, y \in S$ are $\epsilon$-neighbours iff they are neighbours in the graph $G_\epsilon(S)$. Let $N_\epsilon(x)$ be the set of $\epsilon$-neighbours of $x \in S$.

\subsubsection{Understanding 1st local homology}

We saw earlier in Lemma \ref{lem:local_hom_link} that the 1st local homology group of a vertex $\{x\} \in \vr{\epsilon}{S}$ indicates the number of connected components in $\lk(\{x\})$. In practice this usually happens when there is some $y \in N_\epsilon(x)$ such that for all $y \neq z \in N_\epsilon(x)$ we have that $y \notin N_\epsilon(z)$. Hence we usually observe a big connected component of the link and then a few ``lone'' vertices, the existence of which is detected by $H_1^{\{x\}}$. However, there are rare occasions where there is more than one connected component that is not just a point.

Also, we notice that many word vectors have a range of values of $\epsilon$ for which the 1st local homology changes due to new points coming into the link that form ``lone'' connected components for a short range of $\epsilon$-values before they connect to the main component.

The local structure around $\wordv{bank}$ is a good example to look at in order to understand the 1st local homology.

For $\epsilon < 56^\circ$ the vector $\wordv{bank}$ has the local homology of a point since its link is empty. For $\epsilon \in \{56^\circ, \ldots, 65^\circ\}$ the only thing that changes is the 1st local homology, reflecting the fact that we have more word vectors coming into the link, which have no connections with each other. At $\epsilon = 66^\circ$ the 1st local homology is $(\zmod{2})^5$ and the 6 points in the link are vectors of the following words: `deposit', `branch', `institution', `treasury', `depository', `thrift', as seen in Figure \ref{fig:bank_66_word2vec}.

\begin{figure}[h]
	\caption{Link at $\epsilon = 66^\circ$ of $\wordv{bank}$; distance is the geodesic distance to $\wordv{bank}$.}
	\centering
	\includegraphics[width=0.45\textwidth, height=4cm]{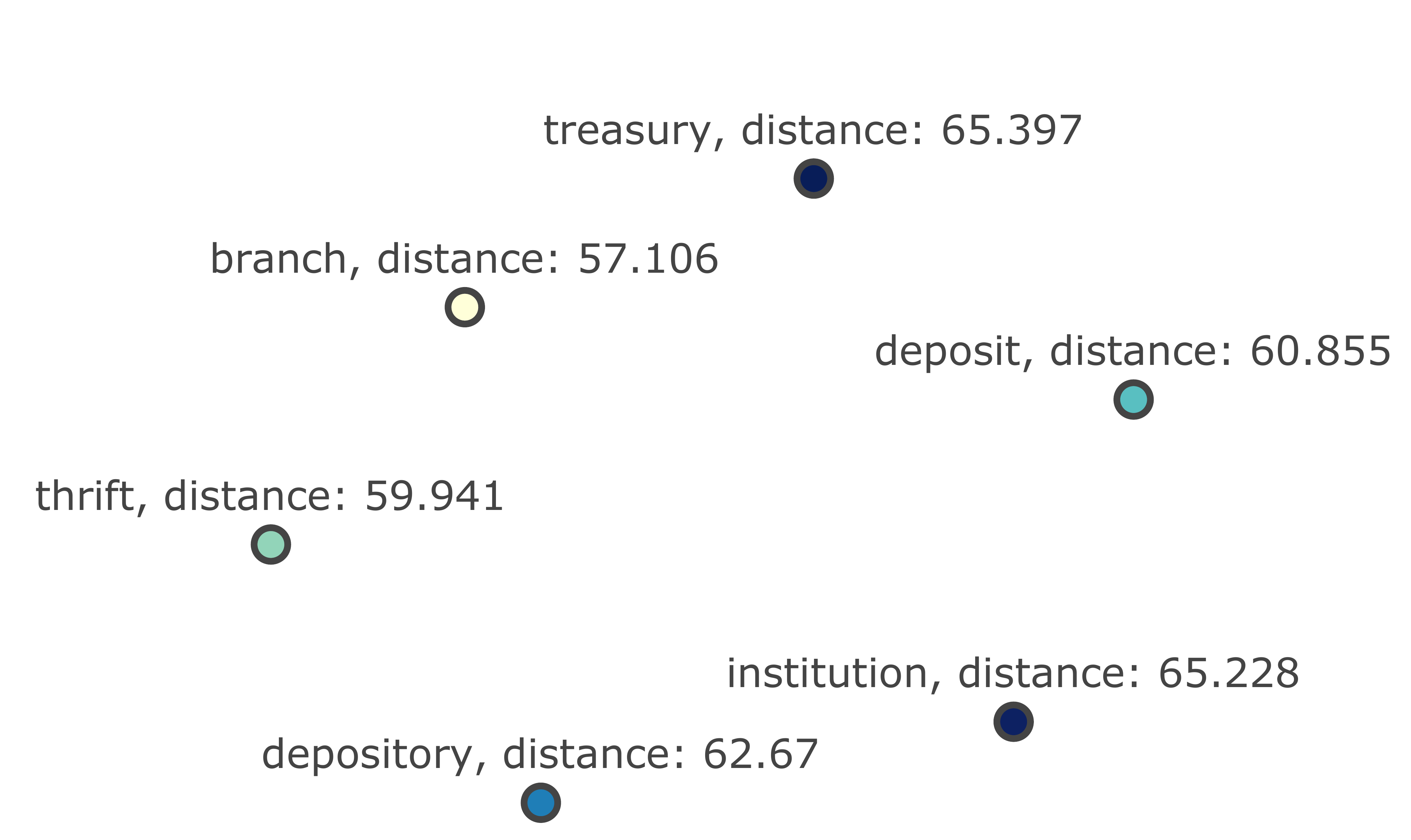}
	\label{fig:bank_66_word2vec}
\end{figure}

When $\epsilon \in \{67^\circ,68^\circ\}$ the 1st local betti number of $\wordv{bank}$ gets smaller reflecting the fact that the 6 words in the link start having connections between each other. At $\epsilon = 69^\circ$ we have another word vector (`supermarket') coming into the link with no connections. At $\epsilon \in \{71^\circ, 72^\circ\}$ the 1st local homology being $\zmod{2}$ reflects the fact that $\wordv{thrift}$ has no connections to any other point in the link but is in the link itself. At $\epsilon \in \{73^\circ, 74^\circ, 75^\circ\}$ the 1st local homology being $\zmod{2}$ reflects the fact that a new word vector (‘syndicate’) becomes part of the link with no connections to the biggest connected component.

When $\epsilon \in \{78^\circ, 79^\circ\}$ the 1st local homology being $\zmod{2}$ reflects the fact that a new connected component -- a 3-clique (‘river’, ‘levee’, ‘shore’) is formed in the link of $\wordv{bank}$, as can be seen in Figure \ref{fig:bank_78_word2vec_cc}. This is one of the rare cases when the smaller connected component of the link is not just a point. Also, the interesting thing is that the big connected component is composed of finance related words but the small (`river', `levee', `shore') is composed of nature related words, confirming our intuition that the fact that `bank' has two unrelated meanings should be reflected by the local homology.

\begin{figure}[h]
	\caption{Link at $\epsilon = 78^\circ$ of $\wordv{bank}$; distance is the geodesic distance to $\wordv{bank}$.}
	\centering
	\includegraphics[width=0.45\textwidth]{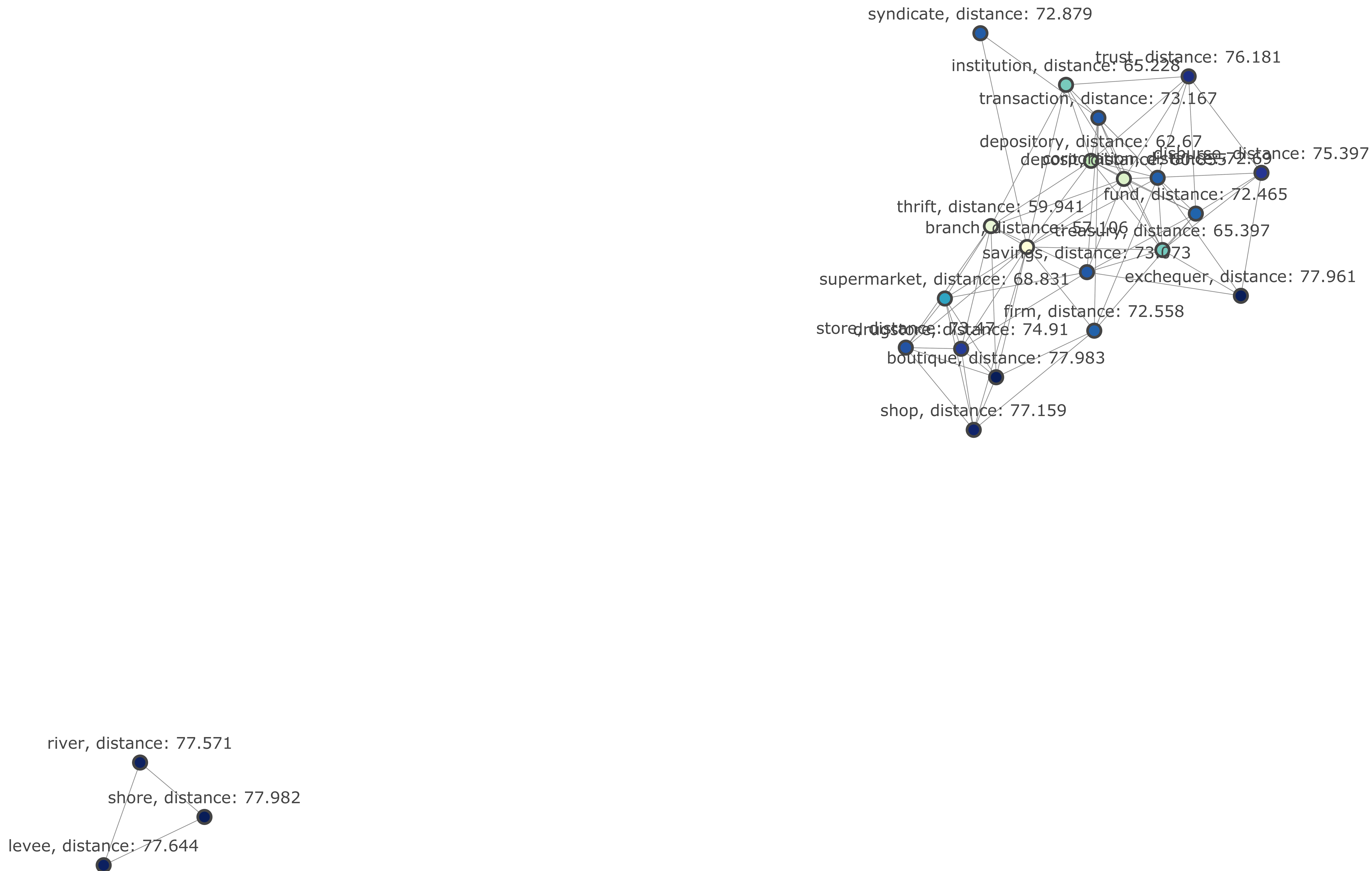}
	\label{fig:bank_78_word2vec_cc}
\end{figure}

\subsubsection{Understanding 2nd local homology}

From Lemma \ref{lem:local_hom_link}, we know that the 2nd local homology of a vertex is isomorphic to the 1st homology of its link. Hence we can say that it counts cycles in the link which do not contain cliques. The reason cycles cannot contain cliques is because $\vr{\epsilon}{S}$ is a clique complex of $G_\epsilon(S)$, and since simplices are contractible, they do not contribute to homology. Intuitively speaking, such cycles in the link appear when there are 4 (or more) words, say $v_1, \, v_2, \, v_3, \, v_4$, such that:
\begin{itemize}
	\item $v_1$ and $v_2$ as well as $v_2$ and $v_3$ are used together in the corpus enough to form a connection at \textbf{that particular $\epsilon$ scale}, but not $v_1$ and $v_3$.
	\item $v_1$ and $v_4$ as well as $v_4$ and $v_3$ are used together in the corpus enough to form a connection at \textbf{that particular $\epsilon$ scale}, but not $v_1$ and $v_3$.
\end{itemize}
So $v_1, \, v_2, \, v_3, \, v_4$ form a ``rectangle'' in the graph. An intuitive example where we might see something like this is with words `reservoir', `reserve' and `stash'. $\wordv{reservoir}$ and $\wordv{reserve}$ are close as well as $\wordv{reserve}$ and $\wordv{stash}$ but not $\wordv{reservoir}$ and $\wordv{stash}$. Semantically, it makes sense because we stash things to reserve them and we use reservoirs to reserve water but we do not really stash water in reservoirs.

The local structures around $\wordv{bank}$ and $\wordv{corporation}$ provide some insight into the 2nd local homology, so let us look into them now.

At $\epsilon = 74^\circ$ $\wordv{bank}$ has the 2nd local homology $(\zmod{2})^2$, which reflects the existence of cycles composed of vectors of the following words: 
\begin{itemize}
	\item `depository', `treasury', `savings', `thrift', `institution'.
	\item `depository', `treasury', `fund', `corporation', `institution'.
\end{itemize}
This can be noted in Figure \ref{fig:bank_74_word2vec}.

\begin{figure}[h]
	\caption{A part of the link of $\wordv{bank}$ at $\epsilon = 74^\circ$; distance is the geodesic distance to $\wordv{bank}$.}
	\includegraphics[width=0.45\textwidth]{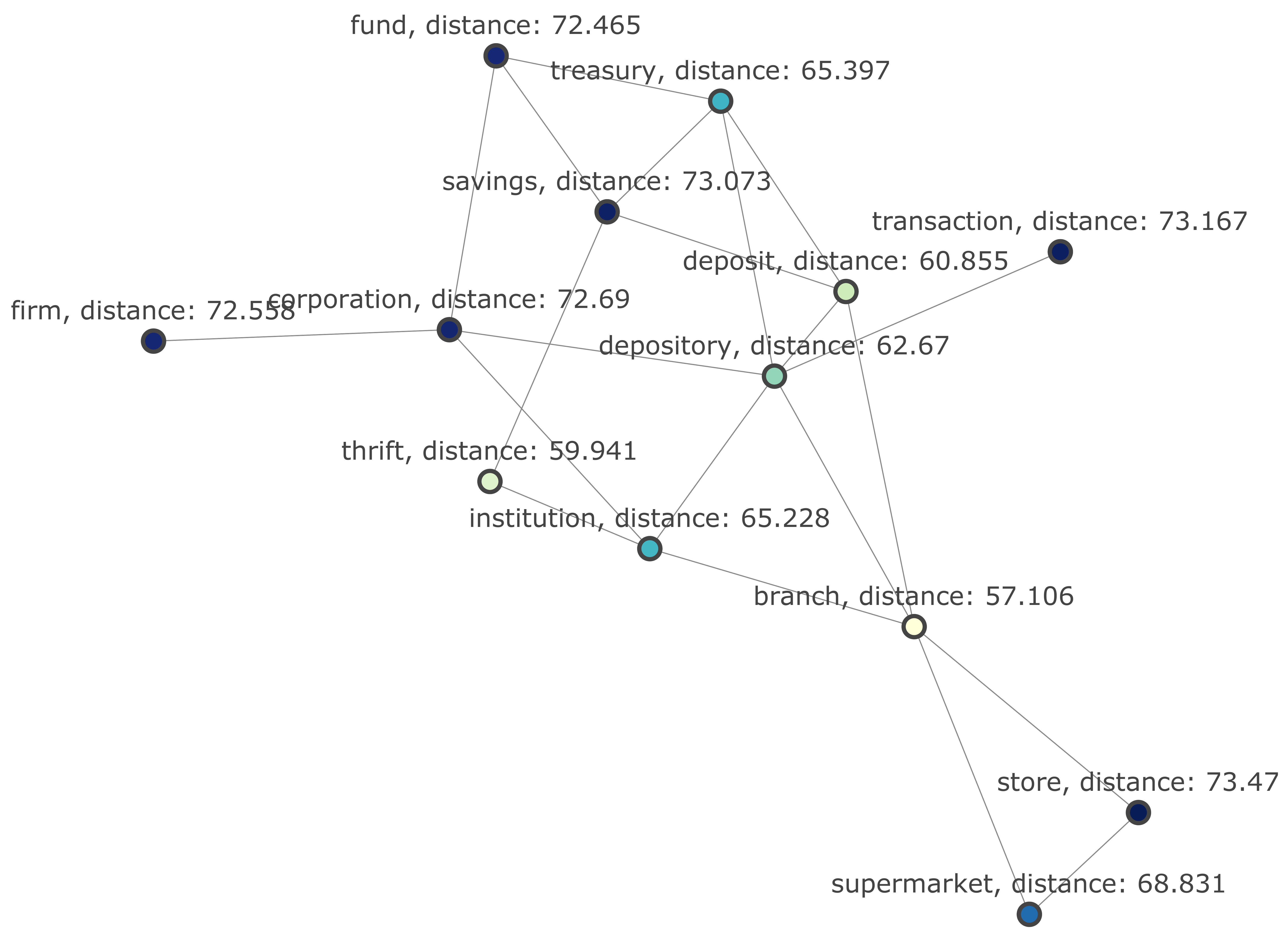}
	\label{fig:bank_74_word2vec}
\end{figure}

At $\epsilon = 75^\circ$ the 2nd local homology of $\wordv{bank}$ is $(\zmod{2})^3$, generated by the following cycles: 
\begin{itemize}
	\item `depository', `treasury', `fund', `corporation'.
	\item `thrift', `branch', `store', `drugstore'.
	\item `depository', `treasury', `savings', `thrift', `institution'.
\end{itemize}
Of course, those representatives of cycles are not unique. In fact, at this scale the cycle `depository', `treasury', `savings', `thrift', `institution' and `depository', `deposit', `savings', `thrift', `institution' represent the same homology class because of the 3-cliques `depository', `deposit', `treasury' and `deposit', `treasury', `savings'. All of this can be seen in Figure \ref{fig:bank_75_word2vec}.

\begin{figure}[H]
	\caption{A part of the link of $\wordv{bank}$ at $\epsilon = 75^\circ$; distance is the geodesic distance to $\wordv{bank}$.}
	\centering
	\includegraphics[width=0.45\textwidth]{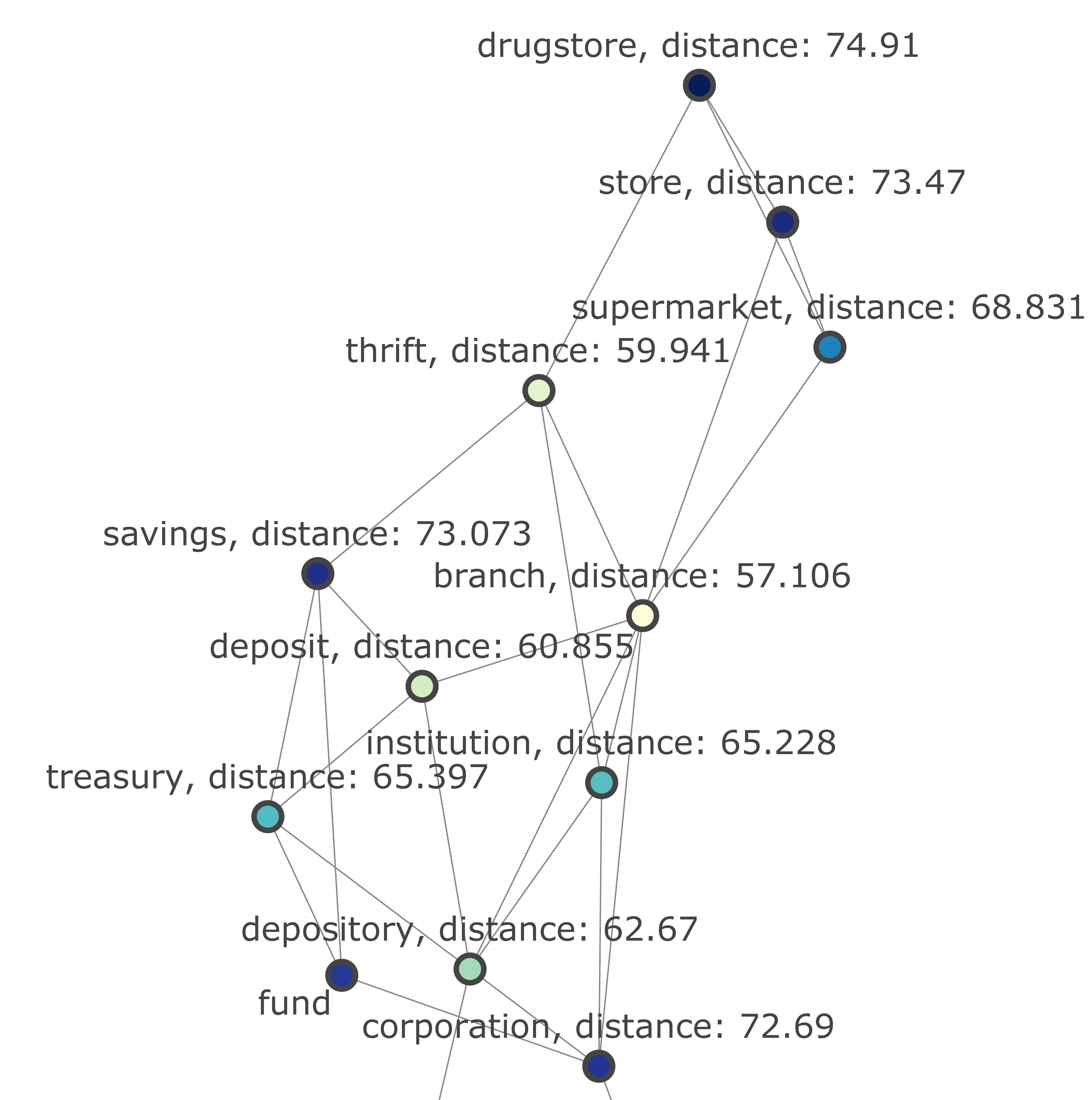}
	\label{fig:bank_75_word2vec}
\end{figure}

At $\epsilon \in \{78^\circ, 79^\circ\}$ we have the 2nd local homology of $\wordv{bank}$ being $\zmod{2}$, which reflects the existence of the cycle `transaction', `syndicate', `branch', `deposit'. We saw the word vector of `syndicate' earlier as a ``lone'' connected component in the link and now, still being relatively far away from a lot of points, at this scale it connects only to a couple points and hence creates a cycle which we capture by looking at the 2nd local homology. This can be noted in Figure \ref{fig:bank_78_word2vec}.

\begin{figure}[H]
	\caption{A part of the link of $\wordv{bank}$ at $\epsilon = 78^\circ$; distance is the geodesic distance to $\wordv{bank}$.}
	\includegraphics[width=0.45\textwidth]{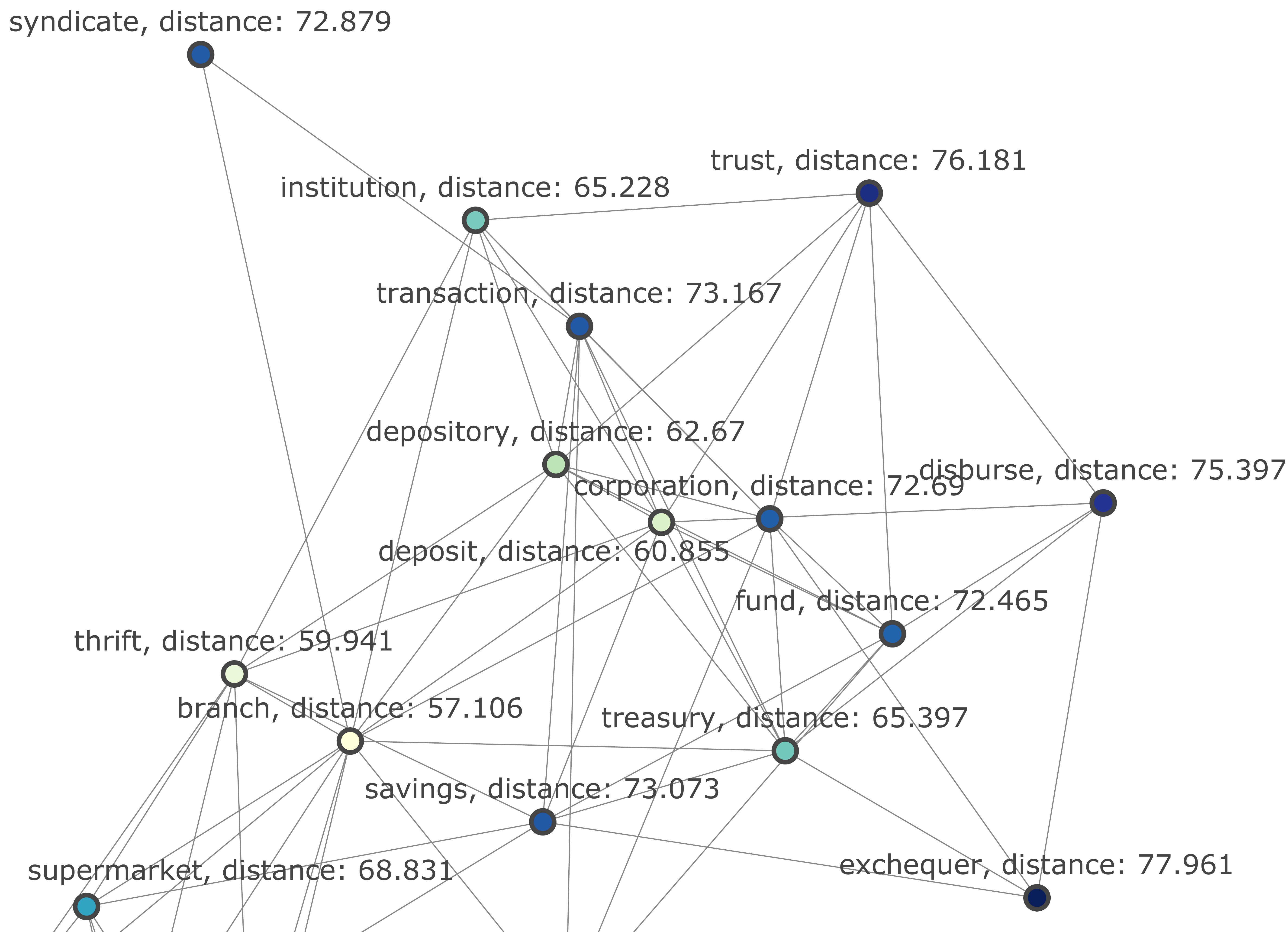}
	\label{fig:bank_78_word2vec}
\end{figure}



Now let us take a look at the vector $\wordv{corporation}$ and its local structure. We will concentrate on the homology when $\epsilon \in \{75^\circ, 76^\circ, 77^\circ\}$.

When $\epsilon = 75^\circ$, the 2nd local homology $\zmod{2}$ generated by the cycle of vectors of `fund', `venture', `firm', `bank', as can be seen in Figure \ref{fig:corporation_75_word2vec}.

\begin{figure}[h]
	\caption{A part of the link of $\wordv{corporation}$ at $\epsilon = 75^\circ$; distance is the geodesic distance to $\wordv{corporation}$.}
	\includegraphics[width=0.45\textwidth]{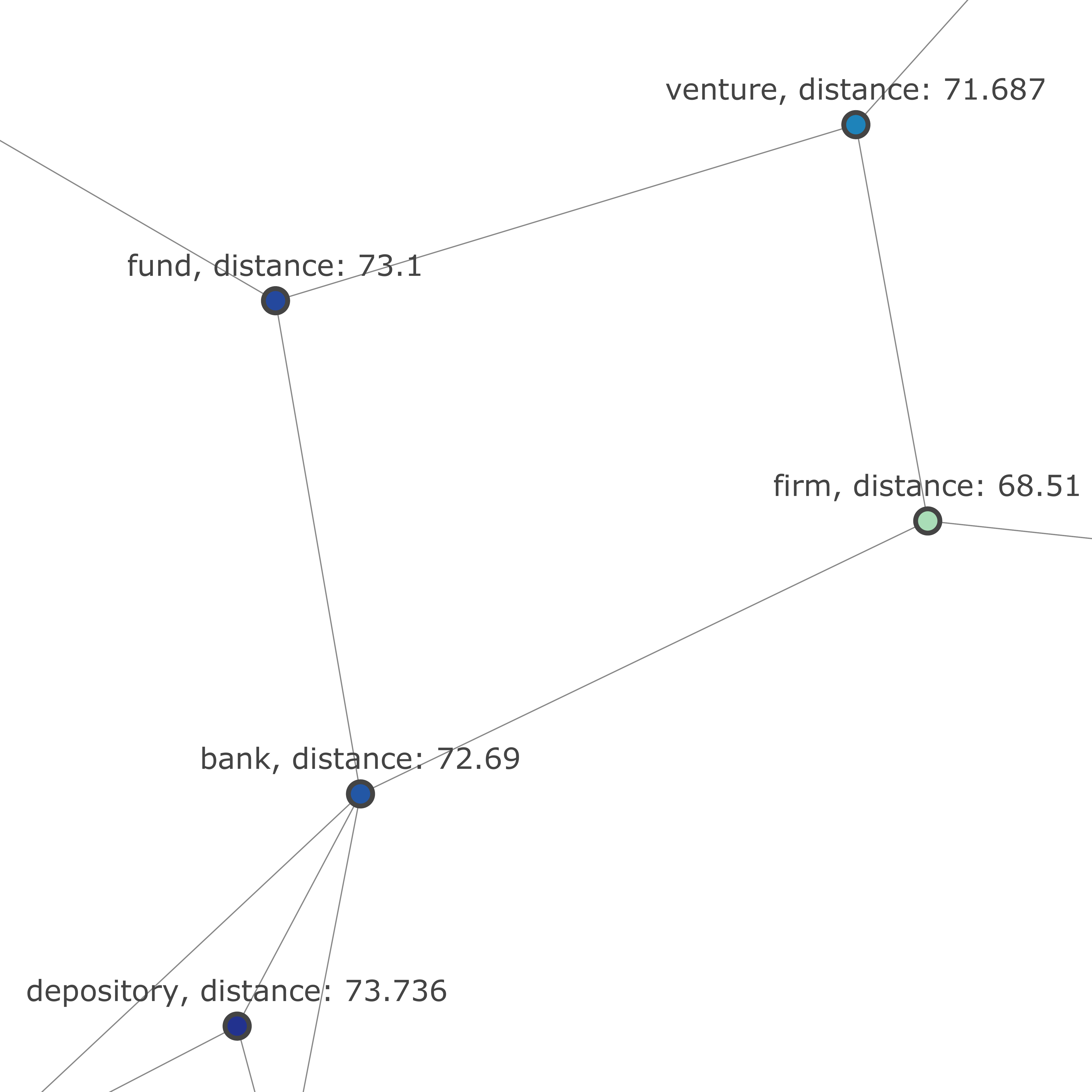}
	\label{fig:corporation_75_word2vec}
\end{figure}

When $\epsilon = 76^\circ$ we suddenly see the 2nd local homology being $(\zmod{2})^7$. By looking at the link we see that the earlier mentioned cycle persists and 6 more cycles appear:
\begin{itemize}
	\item `depot', `branch', `bank', `monopoly'.
	\item `bank', `transaction', `partnership', `organization', `institution'.
	\item `institution', `depository', `transaction', `partnership', `organization'.
	\item `affairs', `institution', `bank', `treasury'.
	\item `branch', `organization', `fund', `bank'.
	\item `fund', `organization', `partnership', `monopoly', `trust'.
\end{itemize}
The cycles can be seen in Figure \ref{fig:corporation_76_word2vec}.

\begin{figure*}
	\caption{The link of $\wordv{corporation}$ at $\epsilon = 76^\circ$; distance is the geodesic distance to $\wordv{corporation}$.}
	\includegraphics[width=\textwidth, height=9cm]{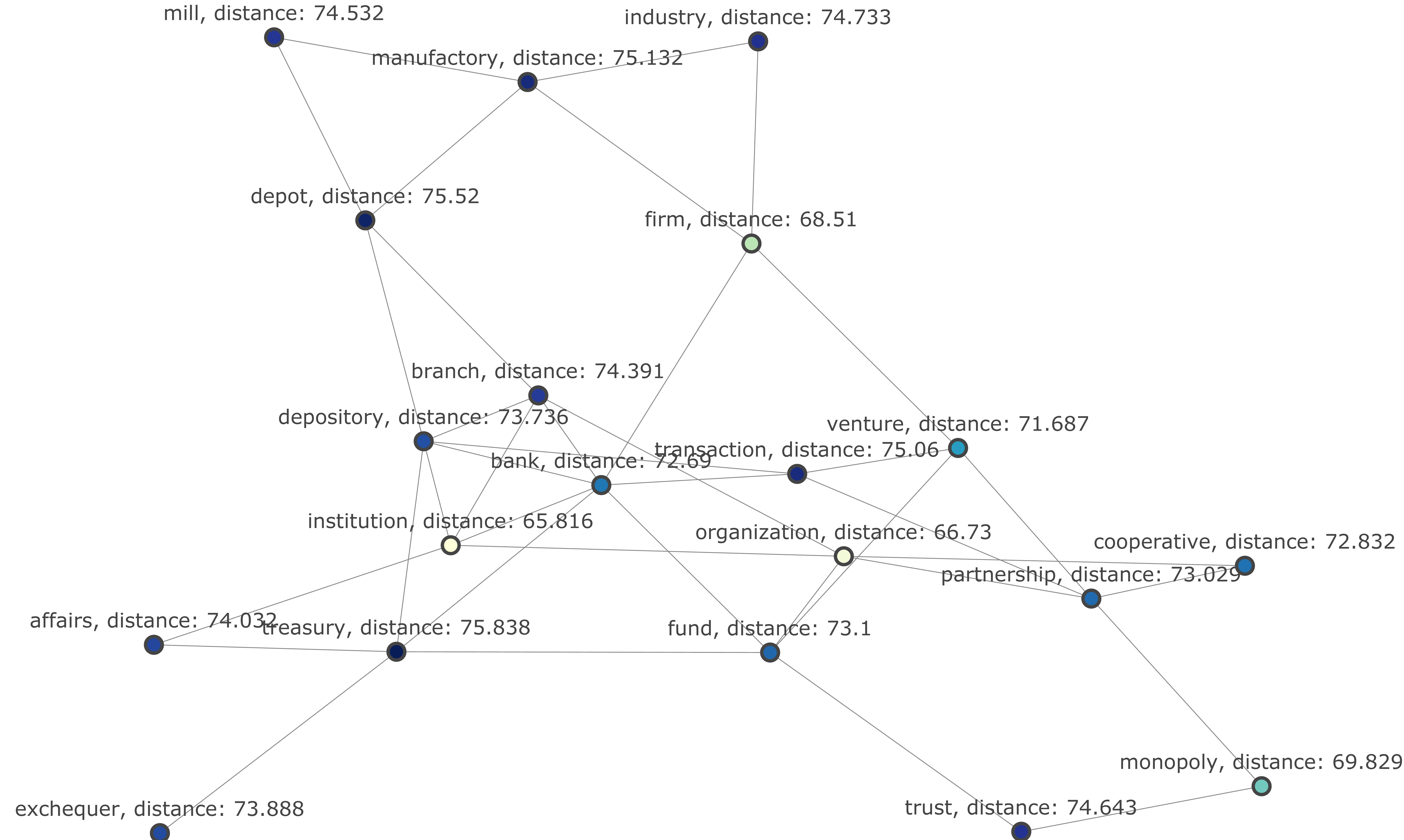}
	\label{fig:corporation_76_word2vec}
\end{figure*}

At $\epsilon = 77^\circ$ the one cycle (`fund', `rganization', `partnership', `monopoly', `trust') gets filled in and hence dies but 4 more cycles come in and we have 10 cycles in total that give rise to the 2nd local homology $(\zmod{2})^{10}$. The 4 new cycles are:
\begin{itemize}
	\item `fund', `trust', `partnership', `organization'.
	\item `manufactory', `establishment', `branch', `depot'.
	\item `bank', `institution', `establishment', `group', `firm'.
	\item `monopoly', `establishment', `institution', `trust'.
\end{itemize}


Afterwards the cycles start dying out and at $\epsilon = 81^\circ$ we observe the 2nd local homology of $\wordv{corporation}$ being trivial.

Higher local homology groups of a vertex correspond to higher homology groups of the link but are harder to interpret intuitively. Of course, the same principle of creating (higher order) cycles works.

\subsection{Results on $D_\text{GloVe}$}

On this dataset we have performed experiments for all $\epsilon \in \{40^\circ, 41^\circ, \ldots, 74^\circ\}$.

In the same way as in the previous subsection, we present a table of words with interesting local homology (Table \ref{glove_localhom}) and then look closer at a few words.

\begin{table}[h]
	\centering
	\caption{Words with interesting homology (GloVe)}
	\begin{tabular}[c]{||c c c||}
		\hline
		Word & Value of $\epsilon$ & Local homology \\ [0.5ex] 
		\hline\hline
		bank&	74&	0 0 1 1 \\
		bank&	68&	0 2 2\\
		bank&	66&	0 1 4\\
		bank&	63&	0 2 1\\
		bank&	62&	0 2 1\\
		corporation&	64&	0 1 2\\
		factory&	70&	0 1 1\\
		fund&	68&	0 1 0 1\\
		house&	68&	0 1 1\\
		invest&	73&	0 0 2 1\\
		invest&	72&	0 0 2 1\\
		invest&	68&	0 1 3 1\\
		river&	68&	0 1 1\\
		river&	66&	0 2 1\\
		stock&	74&	0 1 1\\
		trade&	74&	0 0 1 3\\
		trade&	73&	0 0 1 1\\
		trade&	72&	0 0 1 1\\
		transaction&	68&	0 1 1\\
		waste&	72&	0 1 3\\
		waste&	71&	0 1 2\\
		waste&	68&	0 1 3\\[1ex] 
		\hline
	\end{tabular}
	\label{glove_localhom}
\end{table}

Regarding notation, the same conventions from the last part hold but we will write $S$ for $D_\text{GloVe}$ and \textbf{not} $D_\text{skip-gram}$.

\subsubsection{Understanding 1st local homology}

The intuition behind the 1st local homology from the previous dataset still holds here. Without repeating ourselves we dive straight into a few examples.

Looking at $\wordv{bank}$, we see that in general we observe lower 1st local betti numbers compared to the previous dataset. In the previous dataset we first observed 6 ``lone'' connected components in the link before the vertices started connecting to each other but here it happens for lower $\epsilon$ values. Here, as seen in Figure \ref{fig:bank_58_glove}, we observe a connected component in the link that is not a ``lone'' vertex for quite a low $\epsilon = 58^\circ$, and before that we had only 2 ``lone'' vertices in the link, rather than 6. However, we still observe this phase of ``lone'' vertices coming in, generating 1st local homology for a short period of time and then connecting to the main component. From our observations this generalises to other word vectors as well, not just $\wordv{bank}$. This phase usually encompasses a mid-range of $\epsilon$-values. 

\begin{figure}[h]
	\caption{The link of $\wordv{bank}$ at $\epsilon = 58^\circ$; distance is the geodesic distance to $\wordv{bank}$.}
	\centering
	\includegraphics[width=4cm]{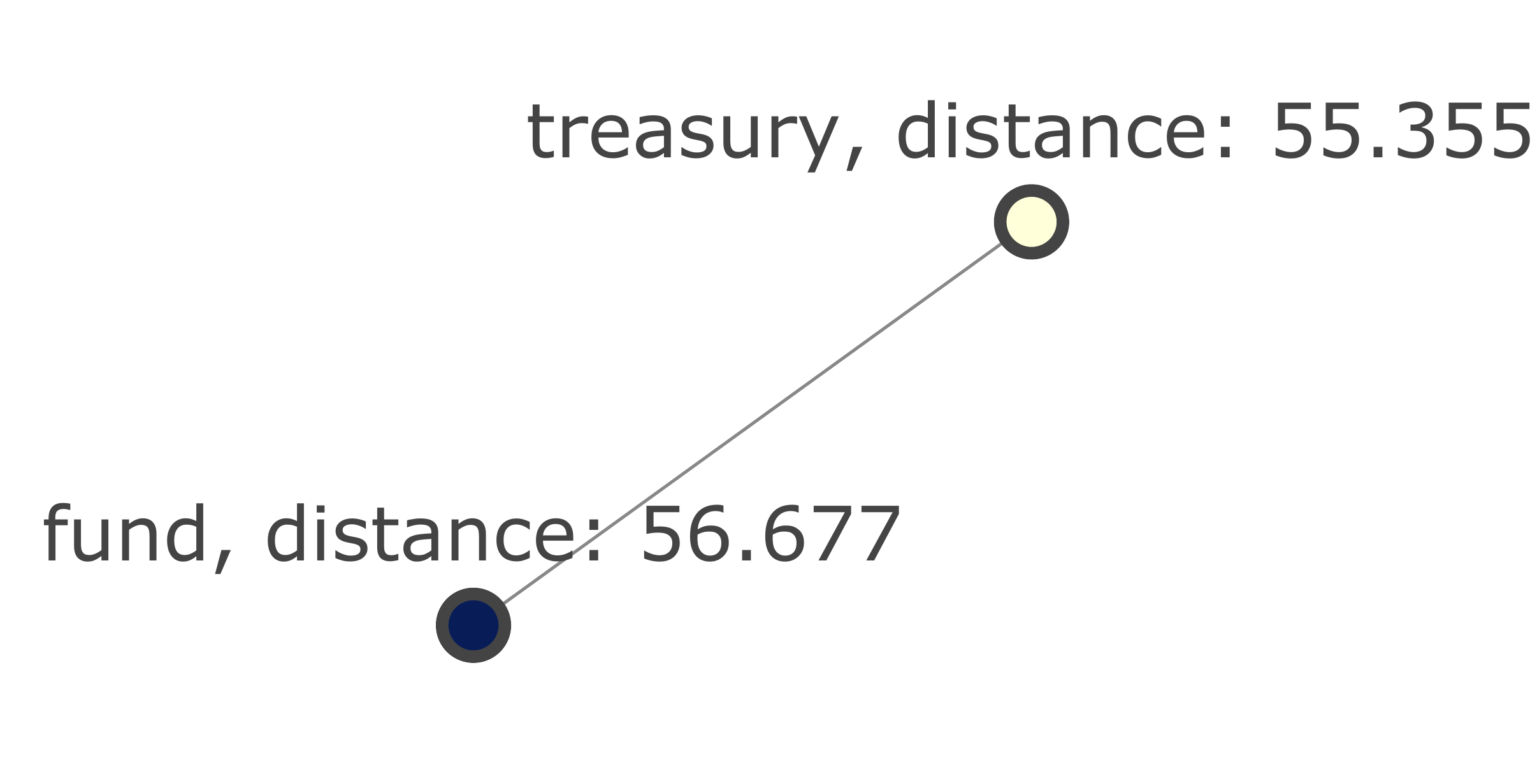}
	\label{fig:bank_58_glove}
\end{figure}

A more interesting observation is that at $\epsilon = 66^\circ$ the link of $\wordv{bank}$ has a connected component consisting of a vector corresponding to `river', which is reflected in the 1st local homology being $\zmod{2}$. This river related connected component persists longer than in the other dataset's case -- until $\epsilon = 70^\circ$ at which point the river related component connects to the main (finance related) one via the vector of `house'. The link of $\wordv{bank}$ before the two connected components merge can be seen in Figure \ref{fig:bank_69_glove}. Similarly to the previous dataset, we see the fact that `bank' has two very different meanings is reflected in its local homology. From a persistence point of view, it is even more pronounced in this dataset.

\begin{figure}[h]
	\caption{The link of $\wordv{bank}$ at $\epsilon = 69^\circ$; distance is the geodesic distance to $\wordv{bank}$.}
	\includegraphics[width=0.45\textwidth]{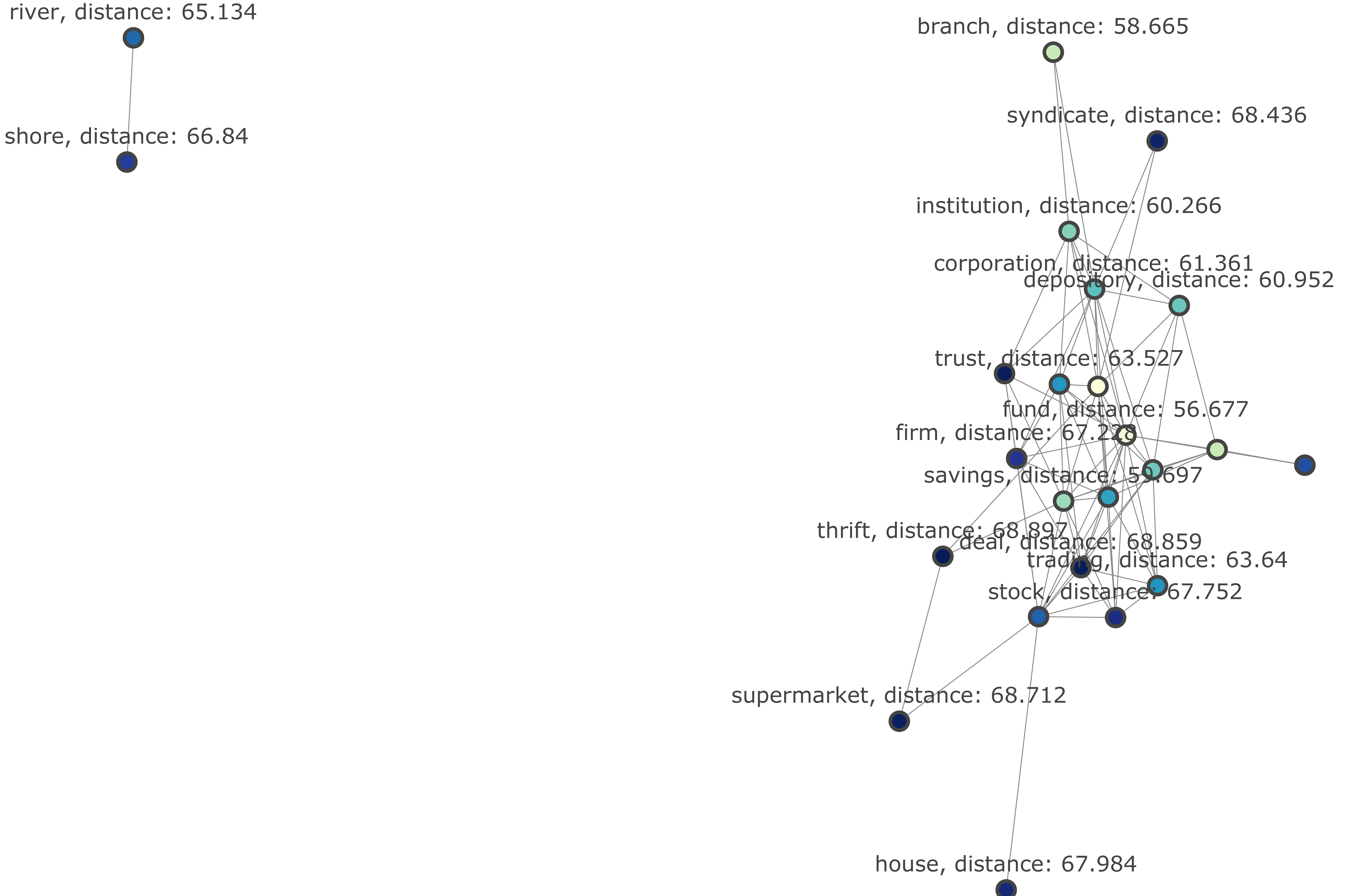}
	\label{fig:bank_69_glove}
\end{figure}

\subsubsection{Understanding 2nd local homology}

The intuition behind the 2nd local homology from the previous dataset still holds here so let us go straight into a few examples.

We observed earlier that the 1st local homology in this dataset is usually of lower dimension compared to the other dataset. We observe a similar thing in the 2nd local homology case as well, though perhaps a little less.

An interesting thing that we observe is a cycle that is common to both datasets. To see this, we can look again at $\wordv{bank}$ when $\epsilon = 62^\circ$. We have the 2nd local homology being $\zmod{2}$ detecting the existence of the cycle `depository', `treasury', `fund', `institution' as seen in Figure \ref{fig:bank_62_glove}, which actually persists for a long time -- it dies at $\epsilon = 68^\circ$. This cycle is homologous to `depository', `treasury', `fund', `corporation', `institution' -- a cycle of vectors of those words we also observed in the previous dataset when $\epsilon = 74^\circ$. Also, by replacing `corporation' with `institution' in the original cycle, we obtain a cycle observed in the previous dataset at $\epsilon = 75^\circ$.

\begin{figure}[h]
	\caption{A part of the link of $\wordv{bank}$ at $\epsilon = 62^\circ$; distance is the geodesic distance to $\wordv{bank}$.}
	\includegraphics[width=0.45\textwidth]{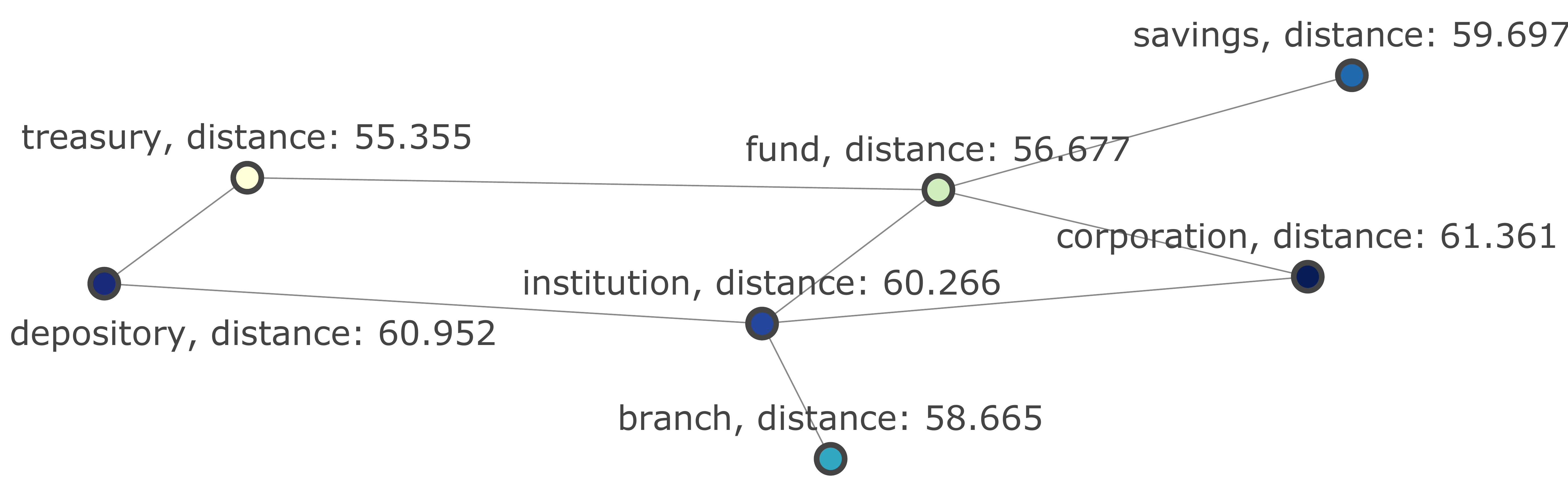}
	\label{fig:bank_62_glove}
\end{figure}

Of course, there are cycles like `invest', `savings', `deposit', `transaction', `trading' appearing at $\epsilon = 65^\circ$, which look like nothing we have seen in the previous dataset -- Figure \ref{fig:bank_65_glove}. In fact, most cycles are ``new'' in the sense that they do not generate non-trivial elements of the local homology groups in the last dataset. Also, we do not see any point with extremely high 2nd local homology -- like `corporation' in $D_\text{skip-gram}$. In this dataset, the vector corresponding to `corporation' has at most 2nd local betti number being 2 and the cycles that generate 2nd local homology are all different from the previous dataset.

\begin{figure}[h]
	\caption{The link of $\wordv{bank}$ at $\epsilon = 65^\circ$; distance is the geodesic distance to $\wordv{bank}$.}
	\includegraphics[width=0.45\textwidth]{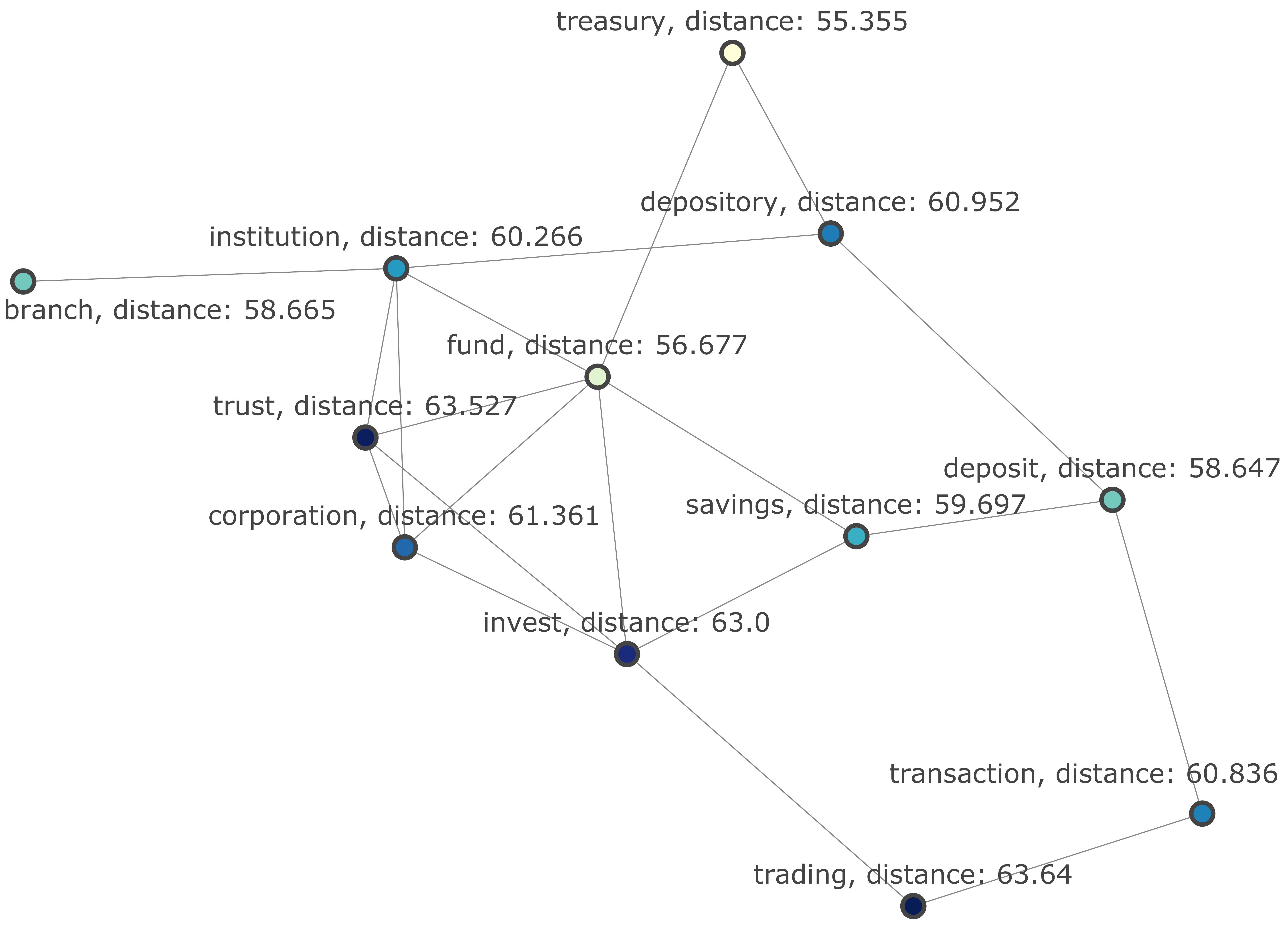}
	\label{fig:bank_65_glove}
\end{figure}

\section{Discussion}

As we can see, local homology does detect interesting structures in the datasets that we are using. Even though the generators of local homology groups are vastly different in the two datasets, we see that the fact that `bank' is a homonym can be seen in both of them by looking at the generators of 1st local homology. By no means is this conclusive but it does suggest that local homology carries relevant information to the problem of word sense disambiguation and therefore this approach is worth exploring further.

However, there still are a few problems with the algorithm. As it was noted earlier, the clusters have little structure. There can be a few reasons for this:
\begin{itemize}
	\item The datasets are too small.
	\item the VR-complex is not the right one to use in this case.
	\item The condition of local homology preserving path for clustering is too limiting.\footnote{as opposed to the earlier proposed condition based on persistent local homology}
	\item The algorithm is not robust to noise and data is noisy.
\end{itemize}
It would be useful to run the algorithm for bigger datasets and this is definitely a part of the future work. Also, it would be very beneficial to better understand how the Vietoris Rips construction affects local homology and to compare how the results change if we use different constructions (e.g. the \v{C}ech complex). Again, due to computational limitations, we used the computationally cheaper construction, which is the VR-complex. Regarding robustness, it is easy to come up with homotopy equivalent simplicial complexes that have different local homologies, which means that the usual guarantees that come from the Nerve Theorem\footnote{statement of the theorem can be found in \citep[Corollary 4G.3.]{hatcher}} do not hold here. In fact, given a value of $\epsilon$, small perturbations can easily change the local homology of many simplices in $\vr{\epsilon}{S}$. As it was suggested earlier, one could change the setting of the algorithm and look at persistent local homology instead with the hope that it would make the algorithm more robust.

Despite the limitations of this algorithm, we believe that it is beneficial to study word embedding problems using local topological constructions (e.g. local persistent homology) with the hope of coming up with a word-sense disambiguation algorithm.

\section{Conclusions and future work}

In this work we have argued that TDA is a promising framework to study word embeddings. We have also presented and implemented a local homology clustering algorithm and tested it on two datasets coming from word embeddings. We have seen that even though some word vectors exhibit interesting local structure that is captured by local homology, the limiting requirement of isomorphic local homology groups render the algorithm susceptible to noise and produce almost trivial clustering (at least on the datasets considered). As previously discussed, we think that moving to persistent local homology seems promising.

Therefore, we would expect future work to include:
\begin{itemize}
	\item Looking into relaxing the limiting isomorphism condition by
	considering persistent local homology or other local
	constructions.
	\item Exploring how stable such an approach is with respect to noise.
	\item If results are successful on word embedding datasets, introducing this approach as a word sense disambiguation algorithm.
\end{itemize}

\section{Acknowledgements}
I would like to thank my master thesis supervisor Vidit Nanda for many insightful conversations, support and guidance.

\bibliography{bibliography}

\end{document}